\documentclass[12pt]{amsart}

\usepackage{amsmath}
\usepackage{amssymb}
\usepackage{amsfonts}
\usepackage{amsthm}
\usepackage{enumerate}
\usepackage{hyperref}
\usepackage{color}
\usepackage{graphicx}
\graphicspath{ {images/} }

\theoremstyle{plain}
\newtheorem{thm}{Theorem}[section]

\newtheorem{prop}[thm]{Proposition}
\newtheorem{lem}[thm]{Lemma}
\newtheorem{cor}[thm]{Corollary}

\theoremstyle{definition}

\newtheorem{rem}[thm]{Remark}

\newtheorem{dfns-rems}[thm]{Definitions and Remarks}
\newtheorem{notas-rems}[thm]{Notations and Remarks}
\newtheorem{exmps-rems}[thm]{Examples and Remarks}

\makeatletter
\@namedef{subjclassname@2020}{%
  \textup{2020} Mathematics Subject Classification}
\makeatother

\DeclareMathOperator{\ind-match}{ind-match}

\begin{document}


\title[symbolic powers of cover ideals]{On the Castelnuovo-Mumford regularity of symbolic powers of cover ideals}


\author[S. A. Seyed Fakhari]{S. A. Seyed Fakhari}

\address{S. A. Seyed Fakhari, Departamento de Matem\'aticas\\Universidad de los Andes\\Bogot\'a\\Colombia.}

\email{s.seyedfakhari@uniandes.edu.co}


\begin{abstract}
Assume that $G$ is a graph with cover ideal $J(G)$. For every integer $k\geq 1$, we denote the $k$-th symbolic power of $J(G)$ by $J(G)^{(k)}$. We provide a sharp upper bound for the regularity of $J(G)^{(k)}$ in terms of the star packing number of $G$. Also, for any integer $k\geq 2$, we study the difference between ${\rm reg}(J(G)^{(k)})$ and ${\rm reg}(J(G)^{(k-2)})$. As a consequence, we compute the regularity of $J(G)^{(k)}$ when $G$ is a doubly Cohen-Macaulay graph. Furthermore, we determine ${\rm reg}(J(G)^{(k)})$ if $G$ is either a Cameron-Walker graph or a claw-free graph which has no cycle of length other that $3$ and $5$.
\end{abstract}






\thanks{}


\maketitle


\section{Introduction} \label{sec1}

Let $S=\mathbb{K}[x_1, \dots, x_n]$ be the polynomial ring over a field $\mathbb{K}$ and let $M$ be a graded $S$-module. Suppose that the minimal graded free resolution of $M$ is given by
$$0\rightarrow \cdots \rightarrow \bigoplus_j S(-j)^{\beta _{1,j}(M)}\rightarrow \bigoplus_j S(-j)^{\beta _{0,j}(M)}\rightarrow M\rightarrow 0.$$
The Castelnuovo-Mumford regularity (or simply, regularity) of $M$, denoted by ${\rm reg}(M)$, is defined as
$${\rm reg}(M)={\rm max}\{j-i\mid \beta _{i,j}(M)\neq 0\}.$$

Computing and finding bounds for the regularity of powers of a monomial ideal have been studied by a number of researchers (see for example \cite{b''}, \cite{bbh}, \cite{c'}, \cite{cht}, \cite{ha}, \cite{hv}, \cite{htt}, \cite{js}, \cite{k}, \cite{sy}, \cite{wo}).

There is a natural correspondence between quadratic squarefree monomial ideals of $S$ and finite simple graphs with $n$ vertices. To every simple graph $G$ with vertex set $V(G)=\{x_1, \ldots, x_n\}$ and edge set $E(G)$, one associates its {\it edge ideal} $I=I(G)$ defined by
$$I(G)=\big(x_ix_j: x_ix_j\in E(G)\big).$$In this paper, we investigate the Alexander dual of edge ideals, namely, the ideal$$J(G)=\bigcap_{\{x_i, x_j\}\in E(G)}(x_i, x_j),$$which is called the {\it cover ideal} of $G$. We refer to \cite{hv2} and \cite{s13} for surveys about homological properties of powers of cover ideals.

Let $I$ is a monomial ideal of $S$. For any integer $k\geq 1$, the $k$-th symbolic power of $I$ is denoted by $I^{(k)}$. Dung et al. \cite{dhnt} prove that the limit $\lim_{k\rightarrow\infty}\frac{{\rm reg}(I^{(k)})}{k}$ exists. In the case that $I=J(G)$ is the cover ideal of a graph $G$, the same authors proved in \cite[Theorem 4.6]{dhnt} that this limit is equal to the maximum value of $\frac{|V(G)|+|N_G(A)|-|A|}{2}$ where $A$ is an independent set of $G$ such that $G\setminus N_G[A]$ has no bipartite connected component. We denote this number by $\gamma(G)$. As the first result of this paper, in Proposition \ref{subgraph} we translate a result of Hien and Trung \cite{ht}, about the regularity of symbolic powers of Stanley-Reisner ideals, to the language of cover ideals and show that $${\rm reg}(J(G)^{(k)})\leq (k-1)\gamma(G)+\max\big\{{\rm reg}(J(H))\mid H \ {\rm is \ a \ subgraph \ of} \  G\big\},$$for any integer $k\geq 1$. As a consequence of the above inequality, we prove in Theorem \ref{claw35} that if $G$ is a claw-free graph which has no cycle of length other that $3$ and $5$, then for every integer $k\geq 1$, we have$${\rm reg}(J(G)^{(k)})=k{\rm deg}(J(G))$$where ${\rm deg}(J(G))$ denotes the maximal degree of minimal monomial generators of $J(G)$.

Let $G$ be a graph with $n$ vertices and let $\alpha_2(G)$ denotes the star packing number of $G$ (see Section \ref{sec2} for the definition of star packing number).  In \cite{fhm1}, Fouli, H${\rm \grave{a}}$ and Morey determined a combinatorial lower bound for the depth of edge ideals of graphs. Indeed, they proved that for every graph $G$, the inequality$${\rm depth}(S/I(G))\geq \alpha_2(G)$$holds (see \cite[Corollary 2.2]{s10} for an alternative proof). Using Auslander-Buchsbaum formula, we deduce that$${\rm pd}(S/I(G))\leq n-\alpha_2(G)$$where ${\rm pd}(S/I(G))$ denotes the projective dimension of $S/I(G)$. Hence, using Terai's theorem \cite[Propostion 8.1.10]{hh}, we obtain the following result.

\begin{prop} \label{star1}
For any graph $G$ with $n$ vertices, we have$${\rm reg}(J(G))\leq n-\alpha_2(G).$$
\end{prop}

In Theorem \ref{star}, we extend the assertion of Proposition \ref{star1} to symbolic powers of $J(G)$, by proving the inequality$${\rm reg}(J(G)^{(k)})\leq (k-1)\gamma(G)+n-\alpha_2(G),$$for each integer $k\geq 1$. The upper bound given by the above inequality for ${\rm reg}(J(G)^{(k)})$ is sharp. Indeed, we will see in Corollary \ref{whisk} that the above inequality is equality if $G$ is a fully clique whiskered graph.

In Section \ref{sec4}, for any integer $k\geq 2$, we study the difference between the regularities of $J(G)^{(k)}$ and $J(G)^{(k-2)}$. In order to estimate this difference, we define $\gamma_0(G)$ to be the maximum value of $\frac{|V(G)|+|N_G(A)|-|A|}{2}$ where $A$ is an independent set of $G$ and any bipartite connected component of $G\setminus N_G[A]$ is an isolated vertex. It is obvious from the definition that $\gamma_0(G)$ is an upper bound for $\gamma(G)$. In Theorem \ref{regsymco}, we show that for any graph $G$ and for any integer $k\geq 2$, the inequality$${\rm reg}(J(G)^{(k)})\leq 2\gamma_0(G)+{\rm reg}(J(G)^{(k-2)})$$holds. As a consequence, we deduce in Corollary \ref{even} that$${\rm reg}(J(G)^{(k)})\leq k\gamma_0(G),$$for each graph $G$ and for any even positive integer $k\geq 1$.

As we mentioned above, for any graph $G$, we have $\gamma(G)\leq \gamma_0(G)$. A natural question is for which classes of graphs we have $\gamma(G)= \gamma_0(G)$. We will show in Proposition \ref{onewell} that equality holds if $G$ is a $1$-well-covered graph (see Section \ref{sec2} for the definition of $1$-well-covered graphs). Moreover, we conclude the equality ${\rm reg}(J(G)^{(k)})=k{\rm deg}(J(G))$, for any $1$-well-covered graph $G$ and for any even integer $k\geq 2$. Using Proposition \ref{onewell}, it will be proven in Theorem \ref{doubcm} that for any doubly Cohen-Macaulay graph $G$ and for each integer $k\geq 1$, we have ${\rm reg}(J(G)^{(k)})=k{\rm deg}(J(G))$. In particular, this equality holds if $G$ is a Gorenstein graph.

Section \ref{sec5} is dedicated to the study of regularity of symbolic powers of cover ideals of Cameron-Walker graphs. Selvaraja \cite[Corollary 4.7]{s}, proves that for any Cameron-Walker graph and for each integer $k\geq 1$, the ideal $J(G)^{(k)}$ has linear quotients. Using this result, we show in Theorem \ref{cwreg}
that ${\rm reg}(J(G)^{(k)})=k{\rm deg}(J(G))$, for any Cameron-Walker graph $G$ and for any integer $k\geq 1$.


\section{Preliminaries} \label{sec2}

In this section, we provide the definitions and basic facts which will be used in the next sections. We refer the reader to \cite{hh} for undefined terminologies.

Let $G$ be a simple graph with vertex set $V(G)=\big\{x_1, \ldots,
x_n\big\}$ and edge set $E(G)$ (by abusing the notation, we identify the edges of $G$ with the corresponding quadratic monomial of $S$). For a vertex $x_i\in V(G)$, the {\it neighbor set} of $x_i$ is defined to be the set $N_G(x_i)=\big\{x_j\mid x_ix_j\in E(G)\big\}$. Moreover, the {\it closed neighborhood} of $x_i$ is $N_G[x_i]=N_G(x_i)\cup \{x_i\}$. For a subset $A\subseteq V(G)$, we set $N_G(A)=\bigcup_{x_i\in A}N_G(x_i)$ and $N_G[A]=\bigcup_{x_i\in A}N_G[x_i]$. The {\it degree} of $x_i$, denoted by ${\rm deg}_G(x_i)$ is the cardinality of $N_G(x_i)$. A vertex of degree one is called a {\it leaf}. An edge $e\in E(G)$ is a {\it pendant edge}, if it is incident to a leaf. A {\it pendant triangle} of $G$ is a triangle $T$ of $G$, with the property that exactly two vertices of $T$ have degree two in $G$. A {\it star triangle} is the graph consisting of finitely many triangles sharing exactly one vertex. The graph $G$ is {\it bipartite} if there exists a partition $V(G)=A\cup B$ such that each edge of $G$ is of the form $x_ix_j$ with $x_i\in A$ and $x_j\in B$. If moreover, every vertex of $A$ is adjacent to every vertex of $B$, then we say that $G$ is a {\it complete bipartite} graph and denote it by $K_{a,b}$, where $a=|A|$ and $b=|B|$. A subgraph $H$ of $G$ is called {\it induced} provided that two vertices of $H$ are adjacent if and only if they are adjacent in $G$. For any $A\subseteq V(G)$, we denote the induced subgraph of $G$ on $A$ by $G[A]$. The graph $K_{1,3}$ is called a {\it claw} and the graph $G$ is said to be {\it claw--free} if it has no claw as an induced subgraph. The cycle graph with $n$ vertices will be denoted by $C_n$. For every subset $A\subset V(G)$, the graph $G\setminus A$ is the graph with vertex set $V(G\setminus A)=V(G)\setminus A$ and edge set $E(G\setminus A)=\{e\in E(G)\mid e\cap A=\emptyset\}$. To simplify the notations, we write $G\setminus x_i$, instead of $G\setminus \{x_i\}$. A subset $W$ of $V(G)$ is said to be an {\it independent subset} of $G$ if there are no edges among the vertices of $W$. The cardinality of the largest independent subset of $G$ is the {\it independence number} of $G$ and is denoted by $i(G)$.
A subset $C$ of $V(G)$ is called a {\it vertex cover} of $G$ if every edge of $G$ is incident to at least one vertex of $C$. A vertex cover $C$ is called a {\it minimal vertex cover} of $G$ if no proper subset of $C$ is a vertex cover of $G$. Note that $C$ is a minimal
vertex cover if and only if $V(G)\setminus C$ is a maximal independent set. For a subset $A\subseteq V(G)$, we set ${\mathbf x}_A$ to be the product $\prod_{x_i\in A}x_i$. It is well-known that$$J(G)=\big({\mathbf x}_C \mid C \ {\rm is \ a \ minimal \ vertex \ cover \ of} \ G\big).$$

The graph $G$ is called {\it well-covered} if all minimal vertex covers of $G$ have the same cardinality. We say that $G$ is {\it 1-well-covered graph} if it is well-covered and moreover, for any vertex $x_i\in V(G)$ the graph $G\setminus x_i$ is well-covered with $i(G\setminus x_i)=i(G)$. The graph $G$ is called {\it Cohen-Macaulay}, {\it sequentially Cohen-Macaulay} or {\it Gorenstein} if the ring $S/I(G)$ has the same property. Furthermore, $G$ is said to be {\it doubly Cohen-Macaulay} if it is Cohen-Macaulay and for any vertex $x_i\in V(G)$ the graph $G\setminus x_i$ is Cohen-Macaulay with $i(G\setminus x_i)=i(G)$.

Let $G$ be a graph. A subset $M\subseteq E(G)$ is a {\it matching} if $e\cap e'=\emptyset$, for every pair of edges $e, e'\in M$. The cardinality of the largest matching of $G$ is called the {\it matching number} of $G$ and is denoted by ${\rm match}(G)$. A matching $M$ of $G$ is an {\it induced matching} of $G$ if for every pair of edges $e, e'\in M$, there is no edge $f\in E(G)\setminus M$ with $f\subset e\cup e'$. The cardinality of the largest induced matching of $G$ is the {\it induced  matching number} of $G$ and is denoted by $\ind-match(G)$.

The graph $G$ is said to be a Cameron-Walker graph if ${\rm match}(G)=\ind-match(G)$. It is clear that a graph is Cameron-Walker if and only if all its connected components are Cameron-Walker. By \cite[Theorem 1]{cw} (see also \cite[Remark 0.1]{hhko}), a connected graph $G$ is a Cameron-Walker graph if and only if

$\bullet$ it is a star graph, or

$\bullet$ it is a star triangle, or

$\bullet$ it consists of a connected bipartite graph $H$ by vertex partition $V(H)=X\cup Y$ with the property that  there is at least one pendant edge attached  to each  vertex of $X$ and there may be some pendant triangles attached to each vertex of $Y$.

Let $G$ be a graph and let $x_i$ be a vertex of $G$. The subgraph ${\rm St}(x_i)$ of $G$ with vertex set $N_G[x_i]$ and edge set $\{x_ix_j\mid x_j\in N_G(x_i)\}$ is called a {\it star with center $x_i$}. A {\it star packing} of $G$ is a family $\mathcal{X}$ of stars in $G$ which are pairwise disjoint, i.e., $V({\rm St}(x_i))\cap V({\rm St}(x_j))=\emptyset$, for ${\rm St}(x_i), {\rm St}(x_j)\in \mathcal{X}$ with $x_i\neq x_j$. The quantity$$\max\big\{|\mathcal{X}|\, |\, \mathcal{X} \ {\rm is \ a \ star \ packing \ of} \ G\big\}$$ is called the {\it star packing number} of $G$. Following \cite{fhm1}, we denote the star packing number of $G$ by $\alpha_2(G)$.

A {\it simplicial complex} $\Delta$ on the set of vertices $[n]:=\{1,
\ldots,n\}$ is a collection of subsets of $[n]$ which is closed under
taking subsets; that is, if $F \in \Delta$ and $F'\subseteq F$, then also
$F'\in\Delta$. Every element $F\in\Delta$ is called a {\it face} of
$\Delta$. A {\it facet} of $\Delta$ is a maximal face
of $\Delta$ with respect to inclusion. The
set of facets of $\Delta$ will be denoted by $\mathcal{F}(\Delta)$. The {\it
Stanley--Reisner} ideal of $\Delta$ is the ideal $I_{
\Delta}$ of $S$ which is defined as$$I_{\Delta}=\big(\prod_{x_i\in F}x_i\mid F\subseteq [n] \ {\rm and} \ F\notin \Delta\big).$$By \cite[Lemma 1.5.4]{hh},$$I_{\Delta}=\bigcap_{F\in \mathcal{F}(\Delta)}(x_i \mid i\in [n]\setminus F).$$

Assume that $I$ is an ideal of $S$ and ${\rm Min}(I)$ is the set of minimal primes of $I$. For every integer $k\geq 1$, the $k$-th {\it symbolic power} of $I$,
denoted by $I^{(k)}$, is defined to be$$I^{(k)}=\bigcap_{\frak{p}\in {\rm Min}(I)} {\rm Ker}(S\rightarrow (S/I^k)_{\frak{p}}).$$Let $I$ be a squarefree monomial ideal with irredundant
primary decomposition $$I=\frak{p}_1\cap\ldots\cap\frak{p}_r,$$ where every
$\frak{p}_i$ is a prime ideal generated by a subset of the variables. It follows from \cite[Proposition 1.4.4]{hh} that for every integer $k\geq 1$, $$I^{(k)}=\frak{p}_1^k\cap\ldots\cap
\frak{p}_r^k.$$In particular, for every graph $G$, we have $$J(G)^{(k)}=\bigcap_{x_ix_j\in E(G)}(x_i, x_j)^k,$$for every integer $k\geq 1$.

For a monomial ideal $I$ its {\it degree}, denoted by ${\rm deg}(I)$ is the maximum degree of its minimal monomial generators. Thus, in particular, ${\rm deg}(J(G))$ is the cardinality of the largest minimal vertex cover of $G$. The monomial ideal $I$ is said to have a {\it linear resolution} if it is generated in a single degree and ${\rm reg}(I)={\rm deg}(I)$. A characterization of graphs $G$ such that $J(G)^{(k)}$ has a linear resolution for some (equivalently, for all) integer $k\geq 2$ is provided in \cite{s11}.

For a monomial ideal $I$ and for an integer $j$, let $I_{\langle j\rangle}$
denote the ideal generated by all monomials of degree $j$
belonging to $I$. The monomial ideal $I$ is called {\it
componentwise linear} if $I_{\langle j\rangle}$ has a linear resolution for
all $j$. We know from \cite[Corollary 8.2.14]{hh} that if $I$ is a componentwise linear ideal, then ${\rm reg}(I)={\rm deg}(I)$. One says that the monomial ideal $I$ has {\it linear quotients} if the minimal monomial generators of $I$ can be ordered as $u_1, u_2, \ldots, u_m$ such that for every $2\leq i\leq m$, the ideal $(u_1, \ldots, u_{i-1}):u_i$ is generated by a subset of the variables. By \cite[Theorem 8.2.15]{hh} any monomial ideal $I$ which has linear quotients is componentwise linear and so, ${\rm reg}(I)={\rm deg}(I)$.


\section{Regularity, subgraphs and star packing numbers} \label{sec3}

As we mentioned in the Introduction, Dung et al. \cite{dhnt} prove that for any monomial ideal $I$, the limit $\lim_{k\rightarrow\infty}\frac{{\rm reg}(I^{(k)})}{k}$ exists. Moreover, we know from \cite[Theorem 4.6]{dhnt} that if $I=J(G)$ is the cover ideal of a graph $G$, then this limit is equal to
\begin{align*}
\gamma(G):=\max\bigg\{& \frac{|V(G)|+|N_G(A)|-|A|}{2}\mid A \ {\rm is \ an \ independent \ set \ of} \ G \ {\rm and} \\ & \ G\setminus N_G[A] \ {\rm has \ no \ bipartite \ connected \ component}\bigg\}.
\end{align*}
The goal of this section is to provide a sharp upper bound for the regularity of $J(G)^{(k)}$ in terms of $\gamma(G)$ and $\alpha_2(G)$. Meanwhile, we also compute the regularity of symbolic powers of cover ideals of claw-free graphs which have no cycle of length other that $3$ and $5$. Our first result is the following proposition which is essentially a special case of \cite[Theorem 2.3]{ht}.

\begin{prop} \label{subgraph}
Let $G$ be a graph. Then for every integer $k\geq 1$, we have$${\rm reg}(J(G)^{(k)})\leq (k-1)\gamma(G)+\max\big\{{\rm reg}(J(H))\mid H \ {\rm is \ a \ subgraph \ of} \  G\big\}.$$
\end{prop}

\begin{proof}
As $J(G)$ is a squarefree monomial ideal, there is simplicial complex $\Delta$ with $J(G)=I_{\Delta}$. We know from \cite[Theorem 2.3]{ht} that ${\rm reg}(J(G)^{(k)})$ is bounded above by$$(k-1)\gamma(G)+\max\big\{{\rm reg}(I_{\Gamma})\mid \Gamma \ {\rm is \ a \ subcomplex \ of} \  \Delta \ {\rm with} \ \mathcal{F}(\Gamma)\subseteq \mathcal{F}(\Delta)\big\}.$$Thus, it is enough to show that for any subcomplex of $\Delta$ with $\mathcal{F}(\Gamma)\subseteq \mathcal{F}(\Delta)$, we have $I_{\Gamma}=J(H)$, for some subgraph $H$ of $G$.

Suppose $V(G)=\{x_1, \ldots, x_n\}$. Note that$$\mathcal{F}(\Delta)=\big\{V(G)\setminus\{x_i, x_j\}\mid x_ix_j\in E(G)\big\}.$$Assume that $\Gamma$ is a subcomplex of $\Delta$ with $\mathcal{F}(\Gamma)\subseteq \mathcal{F}(\Delta)$ and let $H$ be the subgraph of $G$ with $V(H)=V(G)$ and$$E(H)=\big\{x_ix_j\mid V(G)\setminus\{x_i, x_j\}\in \mathcal{F}(\Gamma)\big\}.$$Then one can easily see that $I_{\Gamma}=J(H)$.
\end{proof}

Using Proposition \ref{subgraph}, we are able to compute the regularity of cover ideals of claw-free graphs which have no cycle other than $C_3$ and $C_5$. We first need to compare the degree of cover ideal of a claw-free graph with that of its subgraphs.

\begin{lem} \label{claw}
Let $G$ be a claw-free graph. Then for any subgraph $H$ of $G$, we have ${\rm deg}(J(H))\leq {\rm deg}(J(G))$.
\end{lem}

\begin{proof}
By adding isolated vertices to $H$ (if necessary), we may assume that $V(H)=V(G)$. Let $n$ denote the number of vertices of $G$. By contradiction, suppose that ${\rm deg}(J(H))> {\rm deg}(J(G))$. So, $H$ has a minimal vertex cover $C$ with $|C|> {\rm deg}(J(G))$. Set $A:=V(H)\setminus C$. Then $A$ is a maximal independent set of $H$. Let $B$ be a maximal independent set of $G[A]$. Then $B$ is contained in a maximal independent set $B'$ of $G$. Since $A$ is a maximal independent set of $H$, we conclude that any vertex in $B'\setminus B$ is adjacent in $G$ to at least one vertex in $A\setminus B$. On the other hand, $V(G)\setminus B'$ is a minimal vertex cover of $G$. Therefore,$$n-|B'|\leq {\rm deg}(J(G))< |C|=n-|A|.$$Hence, $|B'|>|A|$ which implies that $|B'\setminus B|>|A\setminus B|$. Thus, there are two vertices $x, y\in B'\setminus B$ which are adjacent in $G$ to a vertex $w\in A\setminus B$. But since $B$ is a maximal independent set of $G[A]$, we deduce that $B\cup \{w\}$ is not an independent set of $G$. Hence, there is a vertex $z\in B$ with $zw\in E(G)$. As a consequence, the vertices $z, w, x, y$ form a claw in $G$ which is a contradiction.
\end{proof}

\begin{rem}
The assertion of Lemma \ref{claw} is not true if $G$ is not claw-free. For instance, assume that $G$ is the graph with vertex set $V(G)=\{x_1, \ldots, x_6\}$ and edge set $E(G)=\{x_1x_2, x_1x_3, x_1x_4, x_2x_5, x_2x_6\}$. Also, let $H$ be the graph obtained from $G$ by deleting the edge $x_1x_2$. Then one can easily check that ${\rm deg}(J(G))=3$ and ${\rm deg}(J(H))=4$.
\end{rem}

In the following theorem, we prove that if $G$ is a claw-free graph which has no cycle of length other than $3$ and $5$, then ${\rm reg}(J(G)^{(k)})$ is as small as possible.

\begin{thm} \label{claw35}
Let $G$ be a claw-free graph which has no cycle other than $C_3$ and $C_5$. Then for every integer $k\geq 1$, we have$${\rm reg}(J(G)^{(k)})=k{\rm deg}(J(G)).$$
\end{thm}

\begin{proof}
We know from \cite[Lemma 3.1]{s2} that ${\rm reg}(J(G)^{(k)})\geq k{\rm deg}(J(G))$. To prove the reverse inequality, note that by \cite[Theorem 4.9]{dhnt}, we have $\gamma(G)={\rm deg}(J(G))$. Thus, using Proposition \ref{subgraph}, it is enough to show that for any subgraph $H$ of $G$, the inequality ${\rm reg}(J(H))\leq {\rm deg}(J(G))$ holds. Observe that $H$ has no cycle other that $C_3$ and $C_5$. Thus, \cite[Theorem 1]{wo2} implies that $H$ is a sequentially Cohen-Macaulay graph. Therefore, we conclude from \cite[Theorem 8.2.20]{hh} that $J(H)$ is componentwise linear. Consequently, it follows from \cite[Corollary 8.2.14]{hh} and Lemma \ref{claw} that$${\rm reg}(J(H))={\rm deg}(J(H))\leq {\rm deg}(J(G)).$$
\end{proof}

In the following theorem, for any graph $G$, we determine a sharp upper bound for ${\rm reg}(J(G)^{(k)})$, in terms of $\gamma(G)$ and $\alpha_2(G)$.

\begin{thm} \label{star}
Let $G$ be a graph with $n$ vertices. Then for every integer $k\geq 1$, we have$${\rm reg}(J(G)^{(k)})\leq (k-1)\gamma(G)+n-\alpha_2(G).$$
\end{thm}

\begin{proof}
According to Proposition \ref{subgraph}, it is enough to show that for any subgraph $H$ of $G$, we have$${\rm reg}(J(H))\leq n-\alpha_2(G).$$By adding isolated vertices to $H$ (if necessary), we may suppose that $V(H)=V(G)$. Let $\mathcal{S}$ be the set of the centers of stars in a largest star packing of $G$. In particular, $|\mathcal{S}|=\alpha_2(G)$. As $E(H)\subseteq E(G)$, the stars in $H$ centered at the vertices of $\mathcal{S}$ form a star packing of $H$. Consequently, $\alpha_2(G)\leq \alpha_2(H)$. It follows from Proposition \ref{star1} that$${\rm reg}(J(H))\leq n-\alpha_2(H)\leq n-\alpha_2(G),$$and this completes the proof.
\end{proof}

The following corollary is a consequence of Theorem \ref{star} and improves \cite[Theorems 3.4, 3.6 and 3.7]{s2}.

\begin{cor} \label{buc}
Let $G$ be a graph with $n$ vertices which belongs to either of the following families:
\begin{itemize}
\item[(i)] bipartite graphs;
\item[(ii)] well-covered graphs;
\item[(iii)] claw-free graphs.
\end{itemize}
Then for every integer $k\geq 1$, we have$${\rm reg}(J(G)^{(k)})\leq (k-1){\rm deg}(J(G))+n-\alpha_2(G).$$
\end{cor}

\begin{proof}
Let $G$ be a graph satisfying the assumption. We know from \cite[Theorem 4.9]{dhnt} that $\gamma(G)={\rm deg}(J(G))$. The assertion now follows from Theorem \ref{star}.
\end{proof}

Cook and Nagel \cite{cn} defined the concept of fully clique-whiskered graphs in the following way. For a
given graph $G$, a subset $W\subseteq V(G)$ is called a {\it clique} of $G$ if every
pair of vertices of $W$ are adjacent in $G$. Let $\pi$ be a partition of
$V(G)$, say $V(G)=W_1\cup\cdots \cup W_t$, such that $W_i$ is a clique of
$G$ for every $1\leq i \leq t$. Then we say that $\pi=\{W_1, \ldots, W_t\}$ is a clique vertex-partition of $G$. Add new vertices $y_1,\ldots,y_t$ and new
edges $xy_i$ for every $x\in W_i$ and every $1\leq i\leq t$. The resulting
graph is called a {\it fully clique-whiskered graph} of $G$, denoted by $G^{\pi}$. The regularity of symbolic powers of cover ideals of fully clique-whiskered graphs is computed by Selvaraja \cite[Corollary 4.11]{s}. Here, we provide an alternative proof for this result. The proof of the following corollary also shows that the bound provided in Theorem \ref{star} for ${\rm reg}(J(G)^{(k)})$ is sharp.

\begin{cor} \label{whisk}
Let $G$ be a graph and suppose that $\pi=\{W_1, \ldots, W_t\}$ is a clique vertex-partition of $G$. Then ${\rm reg}(J(G^{\pi})^{(k)})=k|V(G)|$.
\end{cor}

\begin{proof}
Let $y_1, \ldots, y_t$ be the vertices in $V(G^{\pi})\setminus V(G)$, such that $N_{G^{\pi}}(y_i)=W_i$. We know from \cite[Lemma 3.2]{cn} that $G^{\pi}$ is an well-covered graph with ${\rm deg}(J(G^{\pi})=|V(G^{\pi}|-t=|V(G)|$. Thus, we conclude from \cite[Lemma 3.1]{s2} that ${\rm reg}(J(G^{\pi})^{(k)})\geq k|V(G)|$. On the other hand, Corollary \ref{buc} implies that$${\rm reg}(J(G^{\pi})^{(k)})\leq (k-1)|V(G)|+|V(G^{\pi})|-\alpha_2(G^{\pi}).$$Thus, it is enough to show that $|V(G)|=|V(G^{\pi})|-\alpha_2(G^{\pi})$. This equality is obvious, as the stars centered in $y_1, \ldots, y_t$ form the largest star packing in $G^{\pi}$. In other words, $\alpha_2(G^{\pi})=t=|V(G^{\pi})|-|V(G)|$.
\end{proof}


\section{Difference between Regularities of symbolic powers} \label{sec4}

The goal of this section is to study the difference between regularities of $J(G)^{(k)}$ and $J(G)^{(k-2)}$, for each integer $k\geq 2$, Theorem \ref{regsymco}. As a consequence, in Theorem \ref{doubcm}, we compute ${\rm reg}(J(G)^{(k)})$ when $G$ be a doubly Cohen-Macaulay graph. In order to obtain our results, we need to define a new quantity, denoted by $\gamma_0(G)$, as
\begin{align*}
& \gamma_0(G):=\max\bigg\{\frac{|V(G)|+|N_G(A)|-|A|}{2}\mid A \ {\rm is \ an \ independent \ set \ of} \ G \ {\rm and} \\ & \ {\rm any \ bipartite \ connected \ component \ of} \ G\setminus N_G[A] \ {\rm is \ an \ isolated \ vertex}\bigg\}.
\end{align*}

To prove Theorem \ref{regsymco}, we need some auxiliary lemmas. In the following lemma, we provide a method to determine an upper bound for the regularity of a monomial ideal.

\begin{lem} \label{sub}
Let $I$ be a monomial ideal. Moreover, suppose that $W=\{x_1, \ldots ,x_d\}$ is a (possibly empty) subset of variables of $S$. Then$${\rm reg}(I)\leq\max\Big\{{\rm reg}\big((I+(A)): \mathrm{x}_B\big)+|B|\mid A\cap B=\emptyset, A\cup B=W\Big\}.$$
\end{lem}

\begin{proof}
We use induction on $d$. There is nothing to prove for $d=0$. Therefore, suppose that $d\geq 1$. Set $W' =\{x_1, \ldots, x_{d-1}\}$. We know from the induction hypothesis that
\begin{align*} \tag{1} \label{1}
{\rm reg}(I)\leq\max\Big\{{\rm reg}\big((I+(A)): \mathrm{x}_B\big)+|B|\mid A\cap B=\emptyset, A\cup B=W'\Big\}.
\end{align*}
For every pair of subsets $A, B\subseteq W$ with $A\cap B=\emptyset$ and $A\cup B=W'$, it follows from \cite[Lemma 2.10]{dhs} that
\begin{align*}
& {\rm reg}\big((I+(A)): \mathrm{x}_B\big)+|B|\\ & \leq \max\Big\{{\rm reg}\big(((I+(A)): \mathrm{x}_B), x_d\big)+|B|, {\rm reg}\big(((I+(A)): \mathrm{x}_B): x_d\big)+|B|+1\Big\}\\ & =\max\Big\{{\rm reg}\big((I+(A)+(x_d)): \mathrm{x}_B\big)+|B|, {\rm reg}\big((I+(A)): x_d\mathrm{x}_B\big)+|B|+1\Big\}\\ & =\max\Big\{{\rm reg}\big((I+(A\cup\{x_d\})): \mathrm{x}_B\big)+|B|, {\rm reg}\big((I+(A)): \mathrm{x}_{B\cup\{x_d\}}\big)+|B|+1\Big\}.
\end{align*}
The claim now follows by combining the above inequality and inequality (\ref{1}).
\end{proof}

In order to compute the regularity of $J(G)^{(k)}$ using Lemma \ref{sub}, we need to study the ideals of the form $J(G)^{(k)}+(A)$. The following lemma, provides a useful description for this kind of ideals when $A$ is an independent set of $G$.

\begin{lem} \label{delet}
Let $G$ be a graph with vertex set $V(G)=\{x_1, \ldots, x_n\}$ and suppose $A$ is an independent set of $G$. Set $u=\prod_{x_i\in N_G(A)}x_i$. Then for any integer $k\geq 1$, we have$$J(G)^{(k)}+(A)=u^kJ(G\setminus N_G[A])^{(k)}+(A).$$
\end{lem}

\begin{proof}
We first prove the assertion for $k=1$. Let $C$ be a vertex cover of $G$ with $A\cap C=\emptyset$. Then $N_G(A)\subseteq C$ and $C\setminus N_G(A)$ is a vertex cover of $G\setminus N_G[A]$. This shows that$$J(G)+(A)\subseteq uJ(G\setminus N_G[A])+(A).$$For the reverse inclusion, assume that $D$ is a vertex cover of  $G\setminus N_G[A]$. Then the union $D\cup N_G(A)$ is a vertex cover of $G$. This shows that$$uJ(G\setminus N_G[A])+(A)\subseteq J(G)+(A),$$and completes the proof for $k=1$. Now, suppose that $k\geq 2$. Then
\begin{align*}
& J(G)^{(k)}+(A)=\big(J(G)+(A)\big)^{(k)}+(A)=\big(uJ(G\setminus N_G[A])+(A)\big)^{(k)}+(A)\\ & \big(uJ(G\setminus N_G[A])\big)^{(k)}+(A)=u^kJ(G\setminus N_G[A])^{(k)}+(A),
\end{align*}
where the second equality follows from the case $k=1$, and the last equality follows from the fact that the variables dividing $u$ do not appear in the minimal monomial generators of $J(G\setminus N_G[A])$.
\end{proof}

In the following lemma, we investigate the relation between ${\rm reg}(J(G\setminus N_G[A])$ and ${\rm reg}(J(G))$ when $A$ is an independent set of $G$.

\begin{lem} \label{regcol}
Let $G$ be a graph and suppose $A$ is an independent set of $G$. Then$${\rm reg}(J(G\setminus N_G[A])\leq {\rm reg}(J(G))-|N_G(A)|.$$
\end{lem}

\begin{proof}
We know from \cite[Corollary 1.3]{r2} that$${\rm depth}\big(S/(I(G):\mathrm{x}_A)\big)\geq {\rm depth}(S/I(G)).$$In other words$${\rm depth}\big(S/(I(G\setminus N_G[A])+(N_G(A))\big)\geq {\rm depth}(S/I(G)).$$Set $S'=\mathbb{K}\big[x_i\mid 1\leq i\leq n, x_i\notin N_G[A]\big]$. It follows from the above inequality that$${\rm depth}\big(S'/(I(G\setminus N_G[A])\big)+|A|\geq {\rm depth}(S/I(G)).$$Using the Auslander-Buchsbaum formula we deduce that$$n-|N_G[A]|-{\rm pd}\big(S'/(I(G\setminus N_G[A])\big)+|A|\geq n-{\rm pd}(S/I(G)).$$Equivalently,$${\rm pd}\big(S'/(I(G\setminus N_G[A])\big)+|N_G(A)|\leq {\rm pd}(S/I(G)).$$The assertion now follows from Terai's theorem \cite[Propostion 8.1.10]{hh}.
\end{proof}

Let $I$ be a monomial ideal of
$S$ with minimal monomial generators $u_1,\ldots,u_m$,
where $u_j=\prod_{i=1}^{n}x_i^{a_{i,j}}$, $1\leq j\leq m$. For every $i$
with $1\leq i\leq n$, set$$a_i:=\max\{a_{i,j}\mid 1\leq j\leq m\}$$and$$T:=\mathbb{K}[x_{1,1},x_{1,2},\ldots,x_{1,a_1},x_{2,1},
x_{2,2},\ldots,x_{2,a_2},\ldots,x_{n,1},x_{n,2},\ldots,x_{n,a_n}].$$Let $I^{{\rm pol}}$ be the squarefree
monomial ideal of $T$ with minimal generators $u_1^{{\rm pol}},\ldots,u_m^{{\rm pol}}$, where
$u_j^{{\rm pol}}=\prod_{i=1}^{n}\prod_{p=1}^{a_{i,j}}x_{i,p}$, $1\leq j\leq m$. The ideal $I^{{\rm pol}}$
is called the {\it polarization} of $I$. We know from \cite[Corollary 1.6.3]{hh} that ${\rm reg}(I^{{\rm pol}})={\rm reg}(I)$. In the next lemma, we use polarization to extend the assertion of Lemma \ref{regcol} to symbolic powers.

\begin{lem} \label{regcolsym}
Let $G$ be a graph and suppose $A$ is an independent set of $G$. Then for every integer $k\geq 1$, we have$${\rm reg}\big(J(G\setminus N_G[A])^{(k)}\big)\leq {\rm reg}\big(J(G)^{(k)}\big)-k|N_G(A)|.$$
\end{lem}

\begin{proof}
We know from \cite[Lemma 3.4]{s3} that $(J(G)^{(k)})^{{\rm pol}}$ is the cover ideal of a graph $G_k$ with vertex set $$V(G_k)=\big\{x_{i,p}\mid 1\leq i\leq n \ {\rm and} \  1\leq p\leq k\big\},$$ and edge set
\begin{align*}
E(G_k)=\big\{x_{i,p}x_{j,q}\mid x_ix_j\in E(G) \  {\rm and} \  p+q\leq k+1\big\}.
\end{align*}
It follows from \cite[Corollary 1.6.3]{hh} that ${\rm reg}(J(G)^{(k)})={\rm reg}(J(G_k))$. Set$$A'=\{x_{i,p}\mid x_i\in A, 1\leq p\leq k\}.$$Then $A'$ is an independent set of $G_k$. Therefore, Lemma \ref{regcol} implies that$${\rm reg}(J(G_k\setminus N_{G_k}[A'])\leq {\rm reg}(J(G_k))-|N_{G_k}(A')|={\rm reg}(J(G)^{(k)})-|N_{G_k}(A')|.$$Again, it follows from \cite[Lemma 3.4]{s3} and \cite[Corollary 1.6.3]{hh} that$${\rm reg}(J(G_k\setminus N_{G_k}[A'])={\rm reg}((J(G\setminus N_G[A])^{(k)}).$$Consequently,$${\rm reg}\big((J(G\setminus N_G[A])^{(k)}\big)\leq {\rm reg}(J(G)^{(k)})-|N_{G_k}(A')|.$$The assertion follows by noticing that$$N_{G_k}(A')=\{x_{j,p}\mid x_j\in N_G(A), 1\leq p\leq k\},$$and hence, $|N_{G_k}(A')|=k|N_G(A)|$.
\end{proof}

As a consequence of Lemma \ref{regcolsym}, we can compare $\gamma(G)$ and $\gamma(G\setminus N_G[A])$ when $A$ is an independent set of $G$.

\begin{cor} \label{gammacol}
Let $G$ be a graph and suppose $A$ is an independent set of $G$. Then $$\gamma(G\setminus N_G[A])\leq \gamma(G)-|N_G(A)|.$$
\end{cor}

\begin{proof}
It follows from Lemma \ref{regcolsym} and \cite[Theorems 3.6 and 4.6]{dhnt} that
\begin{align*}
& \gamma(G\setminus N_G[A])=\lim_{k\rightarrow \infty}\frac{{\rm reg}\big(J(G\setminus N_G[A])^{(k)}\big)}{k}\leq \lim_{k\rightarrow \infty}\frac{{\rm reg}\big(J(G)^{(k)}\big)-k|N_G(A)|}{k}\\ & =\lim_{k\rightarrow \infty}\frac{{\rm reg}\big(J(G)^{(k)}\big)}{k}-|N_G(A)|=\gamma(G)-|N_G(A)|.
\end{align*}
\end{proof}

We also need the following lemma in the proof of Theorem \ref{regsymco}.

\begin{lem} \label{gammaind}
Let $G$ be a graph with $n$ vertices. Then for every independent set $A$ of $G$, we have$$n+|N_G(A)|-|A|\leq 2\gamma_0(G).$$
\end{lem}

\begin{proof}
If every bipartite connected component of $G\setminus N_G[A]$ is an isolated vertex, then the assertion follows from the definition of $\gamma_0(G)$. So, suppose that $G\setminus N_G[A]$ has some bipartite connected components $H_1, \ldots, H_t$ which are not isolated vertices. For every integer $i$ with $1\leq i\leq t$, assume that $V(H_i)=X_i\sqcup Y_i$ is the bipartition of the vertex set of $H_i$. Without loss of generality, we may suppose that $|X_i|\leq |Y_i|$ for each $i=1, \ldots, t$. Set $A'=A\cup\bigcup_{i=1}^tX_i$. Then $A'$ is an independent set of $G$, and moreover, every bipartite connected component of $G\setminus N_G[A']$ is an isolated vertex. Therefore, we deduce from the definition of $\gamma_0(G)$ that$$n+|N_G(A')|-|A'|\leq 2\gamma_0(G).$$On the other hand, $N_G(A')=N_G(A)\cup\bigcup_{i=1}^tY_i$. Hence,
\begin{align*}
& n+|N_G(A)|-|A|=n+\big(|N_G(A')|-\sum_{i=1}^t|Y_i|\big)-\big(|A'|-\sum_{i=1}^t|X_i|\big)\\ & \leq n+|N_G(A')|-|A'|\leq 2\gamma_0(G),
\end{align*}
where the first inequality follows from $|X_i|\leq |Y_i|$ .
\end{proof}

We are now ready to prove the first main result of this section.

\begin{thm} \label{regsymco}
For any graph $G$ and for any integer $k\geq 2$, we have$${\rm reg}(J(G)^{(k)})\leq 2\gamma_0(G)+{\rm reg}(J(G)^{(k-2)}).$$
\end{thm}

\begin{proof}
Suppose $V(G)=\{x_1, \ldots, x_n\}$. Without loss of generality, we may assume that $G$ has no isolated vertex. According to Lemma \ref{sub}, it is enough to show that for any $A, B\subseteq V(G)$ with $A\cap B=\emptyset$ and $A\cup B=V(G)$, we have$${\rm reg}\big((J(G)^{(k)}+(A)): \mathrm{x}_B\big)+|B|\leq 2\gamma_0(G)+{\rm reg}(J(G)^{(k-2)}).$$

By \cite[Corollary 4.7]{dhnt}, we have $\gamma_0(G)\geq \gamma(G)\geq n/2$. If $A$ contains two vertices $x_i, x_j$ with $x_ix_j\in E(G)$, then $J(G)\subseteq (A)$. In particular, $\big((J(G)^{(k)}+(A)): \mathrm{x}_B\big)=(A)$ and its regularity is one. Thus,$${\rm reg}\big((J(G)^{(k)}+(A)): \mathrm{x}_B\big)+|B|\leq n-1\leq 2\gamma_0(G)+{\rm reg}(J(G)^{(k-2)}).$$So, suppose that $A$ is an independent set of $G$. Set $u:=\prod_{x_i\in N_G(A)}x_i$. Hence, ${\rm deg}(u)=|N_G(A)|$. By Lemma \ref{delet}, we have
\begin{align*}
& \big((J(G)^{(k)}+(A)): \mathrm{x}_B\big)=\big((u^{k}J(G\setminus N_G[A])^{(k)}+(A)): \mathrm{x}_B\big)\\ & =\big(u^{k}J(G\setminus N_G[A])^{(k)}: \mathrm{x}_B\big)+(A),
\end{align*}
where the second equality follows from $A\cap B=\emptyset$. By definition of $u$, we have $\mathrm{x}_B=u\mathrm{x}_{V(G\setminus N_G[A])}$. Thus, we deduce from \cite[Lemma 3.4]{s5} and the above equalities that$$\big((J(G)^{(k)}+(A)): \mathrm{x}_B\big)=u^{k-1}J(G\setminus N_G[A])^{(k-2)}+(A).$$Consequently, \cite[Theorem 20.2]{p'} implies that
\begin{align*}
& {\rm reg}\big((J(G)^{(k)}+(A)): \mathrm{x}_B\big)={\rm reg}\big(u^{k-1}J(G\setminus N_G[A])^{(k-2)}\big)\\ & =(k-1){\rm deg}(u)+{\rm reg}\big(J(G\setminus N_G[A])^{(k-2)}\big).
\end{align*}
Therefore, Lemma \ref{regcolsym} implies that
\begin{align*}
& {\rm reg}\big((J(G)^{(k)}+(A)): \mathrm{x}_B\big)\leq (k-1){\rm deg}(u)+{\rm reg}\big(J(G)^{(k-2)}\big)-(k-2)|N_G(A)|\\ & =(k-1)|N_G(A)|+{\rm reg}\big(J(G)^{(k-2)}\big)-(k-2)|N_G(A)|\\ & ={\rm reg}\big(J(G)^{(k-2)}\big)+|N_G(A)|
\end{align*}
Finally, using Lemma \ref{gammaind}, we conclude that
\begin{align*}
& {\rm reg}\big((J(G)^{(k)}+(A)): \mathrm{x}_B\big)+|B|\leq {\rm reg}\big(J(G)^{(k-2)}\big)+|N_G(A)|+n-|A|\\ & \leq {\rm reg}\big(J(G)^{(k-2)}\big)+2\gamma_0(G).
\end{align*}
This completes the proof.
\end{proof}

As an immediate consequence of Theorem \ref{regsymco}, we obtain the following corollary.

\begin{cor} \label{oddeven}
Let $G$ be a graph and suppose that $k\geq 2$ is an integer. Then
\begin{itemize}
\item [(i)] ${\rm reg}(J(G)^{(2k-1)})\leq (2k-2)\gamma_0(G)+{\rm reg}(J(G))$, and
\item [(ii)] ${\rm reg}(J(G)^{(2k)})\leq (2k-2)\gamma_0(G)+{\rm reg}(J(G)^{(2)})$.
\end{itemize}
\end{cor}

By theorem \ref{regsymco}, we have ${\rm reg}(J(G)^{(2)})\leq 2\gamma_0(G)$. Thus, Corollary \ref{oddeven} implies the following result.

\begin{cor} \label{even}
For any graph $G$ and for every even positive integer $k$, we have$${\rm reg}(J(G)^{(k)})\leq k\gamma_0(G).$$
\end{cor}

\begin{rem}
It is natural to ask whether the inequality given by Corollary \ref{even} is true if $k$ is an odd positive integer. Unfortunately, this is not the case even if $k=1$. For instance, consider the $4$-cycle graph $C_4$. One can easy check that ${\rm reg}(J(C_4))=3$. But $\gamma_0(C_4)=5/2$. Note that $\gamma(C_4)=2$ and ${\rm reg}(J(C_4)^{(2)})=5$. Therefore, this example also shows that one can not replace $\gamma_0(G)$ by $\gamma(G)$ in Corollary \ref{even}.
\end{rem}

Our next goal in this section is to compute ${\rm reg}(J(G)^{(k)})$ when $G$ is a doubly Cohen-Macaulay graph. We first need the following proposition which states that for any $1$-well-covered graph, the quantities $\gamma(G)$ and $\gamma_0(G)$ coincide. Using this equality, we are able to compute ${\rm reg}(J(G)^{(k)})$ when $k$ is an even positive integer.

\begin{prop} \label{onewell}
Let $G$ be a $1$-well-covered graph. Then $\gamma_0(G)=\gamma(G)={\rm deg}(J(G))$. In particular, for any even integer $k\geq 2$, we have $${\rm reg}(J(G)^{(k)})=k{\rm deg}(J(G)).$$
\end{prop}

\begin{proof}
The equality $\gamma(G)={\rm deg}(J(G))$ is known by \cite[Theorem 4.9]{dhnt}. So, we prove that $\gamma_0(G)=\gamma(G)$. Let $A$ be an independent set of $G$. It is enough to prove that $G\setminus N_G[A]$ has no isolated vertex. By contradiction, assume that $v$ is an isolated vertex of $G\setminus N_G[A]$. Therefore, $i(G\setminus N_G[A]\setminus v)=i(G\setminus N_G[A])-1$. Consequently,
\begin{align*}
& i(G\setminus v)=i(G\setminus v\setminus N_G[A])+|A|=i(G\setminus N_G[A]\setminus v)+|A|\\ & =i(G\setminus N_G[A])+|A|-1=i(G)-1,
\end{align*}
which is a contradiction.

To prove the last part of the proposition, note that by \cite[Lemma 3.1]{s2}, we have ${\rm reg}(J(G)^{(k)})\geq k{\rm deg}(J(G))$. The reverse inequality follows from Corollary \ref{even} and the first part.
\end{proof}

We are now ready to compute the regularity of symbolic powers of cover ideals of doubly Cohen-Macaulay graphs.

\begin{thm} \label{doubcm}
Let $G$ be a doubly Cohen-Macaulay graph. Then for every integer $k\geq 1$, we have$${\rm reg}(J(G)^{(k)})=k{\rm deg}(J(G)).$$
\end{thm}

\begin{proof}
We know from \cite[Lemma 3.1]{s2} that ${\rm reg}(J(G)^{(k)})\geq k{\rm deg}(J(G))$. On the other hand, it follows from the definition that any doubly Cohen-Macaulay graph is $1$-well-covered graph. Thus, Proposition \ref{onewell} implies that $\gamma_0(G)={\rm deg}(J(G))$. On the other hand, Eagon-Reiner theorem \cite[Theorem 8.19]{hh} implies that $J(G)$ has a linear resolution an so, ${\rm reg}(J(G))={\rm deg}(J(G))$. Therefore, by Corollaries \ref{oddeven} and \ref{even}, we have ${\rm reg}(J(G)^{(k)})\leq k{\rm deg}(J(G))$. Hence, the assertion follows.
\end{proof}

The argument in the proof of Theorem \ref{doubcm}, essentially shows that the equality ${\rm reg}(J(G)^{(k)})=k{\rm deg}(J(G))$ is true if $G$ is any Cohen-Macaulay $1$-well-covered graph. As a special case of Theorem \ref{doubcm}, we compute the regularity of symbolic powers of cover ideals of Gorenstein graphs.

\begin{cor}
Let $G$ be a Gorenstein graph. Then for every integer $k\geq 1$, we have$${\rm reg}(J(G)^{(k)})=k{\rm deg}(J(G)).$$
\end{cor}

\begin{proof}
Clearly, we may assume that $G$ has no isolated vertex. Then it follows from \cite[Lemma 1.3]{ht2} that $G$ is a doubly Cohen-Macaulay graph. Thus, Theorem \ref{doubcm} implies that ${\rm reg}(J(G)^{(k)})=k{\rm deg}(J(G))$.
\end{proof}


\section{Cameron-Walker graphs} \label{sec5}

In this section, we compute the regularity of symbolic powers of cover ideals of Cameron-Walker graphs, Theorem \ref{cwreg}. We know from \cite[Corollary 4.7]{s} that for any Cameron-Walker graph $G$, the ideal $J(G)^{(k)}$ has linear quotients. As a consequence, to compute ${\rm reg}(J(G)^{(k)})$, it is sufficient to compute ${\rm deg}(J(G)^{(k)})$. To determine the degree of $J(G)^{(k)}$, we need to compute $\gamma(G)$. This will be done in Proposition \ref{cwgam}. We will see in that proposition that for any Cameron-Walker graph $G$, the equality $\gamma(G)={\rm deg}(G)$ holds. We need a couple of lemmas to prove Proposition \ref{cwgam}.

\begin{lem} \label{cwdel}
Let $G$ be a Cameron-Walker graph. Then for any vertex $v\in V(G)$, the graph $G\setminus N_G[v]$ is Cameron-Walker too.
\end{lem}

\begin{proof}
Without loss of generality, we may assume that $G$ is a connected graph. The assertion can be easily checked if $G$ is a star graph or a star triangle graph. So, suppose that $G$ consists of a connected bipartite graph $H$ with vertex partition $V(H)=X\cup Y$ such that there is at least one pendant edge attached to each vertex of $X$ and there may be some pendant triangles attached to each vertex of $Y$. As $G$ does not have any isolated vertex, we have$${\rm match}(G\setminus N_G[v])\leq {\rm match}(G)-1.$$Consequently, if$$\ind-match(G\setminus N_G[v])\geq \ind-match(G)-1,$$then the assertion follows. In particular, if $v$ is a leaf of $G$, then we are done. So assume that $v$ is not a leaf of $G$. Let $t$ denote the number of triangles of $G$. Suppose $X=\{x_1, \ldots, x_p\}$ and $Y=\{y_1, \ldots, y_q\}$. Also, assume that for each integer $i=1, 2, \ldots, p$, the edges $x_iw_{i1}, \ldots, x_iw_{im_i}$ are the pendant edges attached to $x_i$, and for each integer $j=1, 2, \ldots, q$, the triangles with vertices $\{y_j, z_{j1}, z'_{j1}\}, \ldots, \{y_j, z_{jr_j}, z'_{jr_j}\}$ are the pendant triangles attached to $y_j$. In particular, $r_1+r_2+\cdots +r_q=t$. We know from \cite[Lemma 3.3]{s8} that $\ind-match(G)=p+t$. It is obvious that the set$$M=\{x_1w_{11}, \ldots, x_pw_{p1}\}\cup\{z_{jk}z'_{jk}\mid 1\leq j\leq q, 1\leq k\leq r_j\}$$is an induced matching in $G$ of size $p+t$.

$\bullet$ Suppose $v=z_{jk}$ or $v=z'_{jk}$, for some integers $j$ and $k$ with $1\leq j\leq q$ and $1\leq k\leq r_j$. Then by deleting $N_G[v]$ from $G$, only the edge $z_{jk}z'_{jk}$ will be deleted from $M$. So,$$\ind-match(G\setminus N_G[v])\geq m+t-1=\ind-match(G)-1$$and the assertion follows in this case.

$\bullet$ Suppose $v=x_i$, for some integer $i$ with $1\leq i\leq p$. Then by deleting $N_G[v]$ from $G$, only the edge $x_iw_{i1}$ will be deleted from $M$. So,$$\ind-match(G\setminus N_G[v])\geq m+t-1=\ind-match(G)-1$$and the assertion follows in this case.

$\bullet$ Suppose $v=y_i$, for some integer $j$ with $1\leq j\leq q$. Assume that$$X\cap N_G[v]=\{x_{\ell_1}, \ldots, x_{\ell_h}\},$$for some integer $h$ with $1\leq h\leq p$. Note that by deleting $N_G[v]$ from $G$, the vertices $v, x_{\ell_1}, \ldots, x_{\ell_h}, z_{j1}, \ldots, z_{jr_j}, z'_{j1}, \ldots, z'_{jr_j}$ will be deleted. Moreover, the vertices $w_{\ell_11}, \ldots, w_{\ell_1m_{\ell_1}}, \ldots, w_{\ell_h1}, \ldots, w_{\ell_hm_{\ell_h}}$ are isolated vertices in $G\setminus N_G[v]$. Observe the set of edges$$M':=\{x_{\ell_1}w_{\ell_11}, \ldots, x_{\ell_h}w_{\ell_h1}, z_{j1}z'_{j1}, \ldots, z_{jr_j}z'_{jr_j}\}\subseteq E(G).$$It implies that$${\rm match}(G\setminus N_G[v])\leq {\rm match}(G)-|M'|.$$On the other hand $M\setminus M'$ is an induced matching of $G\setminus N_G[v]$ which shows that $$\ind-match(G\setminus N_G[v])\geq \ind-match(G)-|M'|.$$Therefore, $G\setminus N_G[v]$ is a Cameron-Walker graph.
\end{proof}

The following lemma determines a lower bound for ${\rm deg}(J(G))$ when $G$ is a Cameron-Walker graph.

\begin{lem} \label{cwcov}
Let $G$ be a Cameron-Walker graph without isolated vertices. Then $G$ has a minimal vertex cover with cardinality at least $\frac{|V(G)|}{2}$.
\end{lem}

\begin{proof}
Without loss of generality, we may assume that $G$ is a connected graph. The assertion can be easily checked if $G$ is a star graph or a star triangle graph. So, suppose that $G$ consists of a connected bipartite graph $H$ with vertex partition $V(H)=X\cup Y$ such that there is at least one pendant edge attached to each vertex of $X$ and there may be some pendant triangles attached to each vertex of $Y$. Suppose $X=\{x_1, \ldots, x_p\}$ and $Y=\{y_1, \ldots, y_q\}$. Also, assume that for each integer $i=1, 2, \ldots, p$, the edges $x_iw_{i1}, \ldots, x_iw_{im_i}$ are the pendant edges attached to $x_i$, and for each integer $j=1, 2, \ldots, q$, the triangles with vertices $\{y_j, z_{j1}, z'_{j1}\}, \ldots, \{y_j, z_{jr_j}, z'_{jr_j}\}$ are the pendant triangles attached to $y_j$. Then one can easily check that the set$$\{w_{11}, \ldots, w_{1m_1}, \ldots, w_{p1}, \ldots, w_{pm_p}, y_1, \ldots, y_q, z_{11}, \ldots, z_{1r_1}, \ldots, z_{q1}, \ldots, z_{qr_q}\}$$is a minimal vertex cover of $G$ with the desired property.
\end{proof}

In the following proposition, we compute $\gamma(G)$ for a Cameron-Walker graph $G$.

\begin{prop} \label{cwgam}
For any Cameron-Walker graph $G$, we have $\gamma(G)={\rm deg}(J(G))$.
\end{prop}

\begin{proof}
The assertion follows immediately from \cite[Lemma 4.10]{dhnt} and Lemmata \ref{cwdel} and \ref{cwcov}.
\end{proof}

We are now ready to prove the main result of this section.

\begin{thm} \label{cwreg}
Let $G$ be a Cameron-Walker graph. Then for every integer $k\geq 1$, we have ${\rm reg}(J(G)^{(k)})=k{\rm deg}(J(G))$.
\end{thm}

\begin{proof}
We know from \cite[Lemma 3.1]{s2} that ${\rm reg}(J(G)^{(k)})\geq k{\rm deg}(J(G))$. To prove the reverse inequality, note that by \cite[Corollary 4.7]{s}, the ideal $J(G)^{(k)}$ has linear quotients. Thus, its regularity is equal to the maximum degree of its minimal monomial generators. On the other hand, it follows from \cite[Lemma 4.3]{dhnt} and Proposition \ref{cwgam} that the degree of minimal monomial generators of $J(G)^{(k)}$ is bounded above by $k{\rm deg}(J(G))$. Hence, ${\rm reg}(J(G)^{(k)})=k{\rm deg}(J(G))$.
\end{proof}



\end{document}